\documentclass{icm2010}

\usepackage{hyperref}

\def\eq#1\en{\begin{equation}#1\end{equation}}  
\def\eqalign#1\enalign{
	\begin{align}#1\end{align}
	}
\newcommand{\nnb}	{\nonumber \\} 

\newcommand{\lbeq}[1]  {\label{e:#1}}
\newcommand{\refeq}[1] {\eqref{e:#1}}    
\makeatletter
\newcommand{\labelcounter}[2]{{%
	\stepcounter{#1}
	\protected@write\@auxout{}%
	{\string\newlabel{#2}{{\csname the#1\endcsname}{\thepage}}}%
	{\ref{#2}}
	}}
\makeatother

\newcommand{\Cbold} {{\mathbb C}}  
  
\newcommand{\Ebold} {{\mathbb E}}

\newcommand{\Nbold} {{\mathbb N}}

\newcommand{\Rbold} {{\mathbb R}}

\newcommand{\Zbold} {{\mathbb Z}}

\newcommand{\Bcal}   {\mathcal{B}} 
\newcommand{\Ccal}   {\mathcal{C}}

\newcommand{\Ncal}   {\mathcal{N}} 
 
\newcommand{\Pcal}   {\mathcal{P}}

\newcommand{\Scal}   {\mathcal{S}}

\newcommand{\Vcal}   {\mathcal{V}}

\newcommand{\Vhat} {{\hat{V} }}

\newcommand{\Zd}    {{ {\Zbold}^d }}

\newcommand{\nin}  {{ \not\in }}

\usepackage[dvips]{graphicx}
\usepackage{bbm} 

\usepackage[usenames]{color}

\newcommand{\LT}{{\rm Loc}  }
\newcommand{\LTbar}{\overline{\LT}}

\newcommand{\R}{\Rbold}
\newcommand{\Z}{\Zbold}

\newcommand{\N}{\Nbold}
\newcommand{\C}{\mathbb{C}}
\newcommand{\1}{\mathbbm{1}}

\newcommand{\varphib}{\bar\varphi}
\newcommand{\phib}{\bar\phi}
\newcommand{\psib}{\bar\psi}

\newcommand{\pair}[1]{\langle #1 \rangle}

\newcommand{\Phipol}{\Pi}

\newcommand{\h}{\mathfrak{h}}
\newcommand{\Vpt}{V_{\rm pt}}
\newcommand{\gpt}{g_{\mathrm{pt}}}
\newcommand{\nupt}{\nu_{\mathrm{pt}}}

\newcommand{\zpt}{z_{\mathrm{pt}}}
\newcommand{\lambdapt}{\lambda_{\mathrm{pt}}}
\newcommand{\qpt}{q_{\mathrm{pt}}}

\newcommand{\gbar}{\bar{g}}
\newcommand{\pp}{a}
\newcommand{\qq}{b}
\newcommand{\sigmaa}{\sigma}
\newcommand{\sigmab}{\bar{\sigma}}
\newcommand{\s}{\mathfrak s}

\newcommand{\half}{\textstyle{\frac 12}}

\newcommand{\rpt}{r^{\mathrm{pt}}}

\newtheorem{theorem}{Theorem}[section]
\newtheorem{prop}[theorem]{Proposition}

\theoremstyle{definition}
\newtheorem{definition}{Definition}
\newtheorem{example}{Example}

\title[Self-avoiding walk in four dimensions]
        {Renormalisation group analysis of
        weakly self-avoiding walk in dimensions four and higher}
\author[D. Brydges and G. Slade]{David Brydges and Gordon Slade
}
\contact[db5d@math.ubc.ca]{
Department of Mathematics, \\
University of British Columbia,\\
Vancouver, BC, Canada V6T 1Z2
}
\contact[slade@math.ubc.ca]{
Department of Mathematics, \\
University of British Columbia,\\
Vancouver, BC, Canada V6T 1Z2
}

\begin{document}

\noindent \date{March 15, 2010}

\begin{abstract}
We outline a proof, by a rigorous renormalisation group method, that
the critical two-point function for continuous-time weakly
self-avoiding walk on $\Zd$ decays as $|x|^{-(d-2)}$ in the critical
dimension $d=4$, and also for all $d>4$.
\end{abstract}

\begin{classification}
Primary 82B41; Secondary 60K35.
\end{classification}

\begin{keywords}
self-avoiding walk, Edwards model, renormalization group,
supersymmetry, quantum field theory
\end{keywords}

\maketitle




\section{Introduction}\label{sec:introduction}

We prove $|x|^{-(d-2)}$ decay for the critical two-point function of
the continuous-time weakly self-avoiding walk in dimensions $d \ge 4$.
This is a summary of the ideas and the steps in the proof.  The
details are provided in \cite{BS11}.  The proof is based on a rigorous
renormalisation group argument.  For the case $d>4$, this provides an
approach completely different from the lace expansion methods of
\cite{Hara08,HHS03}.  But our main contribution is that our method
applies also in the case of the \emph{critical} dimension $d=4$, where
lace expansion methods do not apply.

Renormalisation group methods have been applied previously to study
weakly self-avoiding walk on a 4-dimensional \emph{hierarchical}
lattice.  The continuous-time version of the model has been studied in
the series of papers \cite{BEI92,GI95,BI03c,BI03d}; see \cite{BJS03}
for a review.  More recently, a completely different renormalisation
group approach to the discrete-time weakly self-avoiding walk on a
4-dimensional hierarchical lattice has been developed in \cite{HO10}.

The $|x|^{-(d-2)}$ decay for the two-point function for a continuum
4-dimensional Edwards model, with a smoothed delta function, has been
proved in \cite{IM94}; unlike our model, this is not a model of walks
taking nearest neighbour steps in the lattice, but it is expected to
be in the same universality class as our model. The relation between
our model and the Edwards model is discussed in \cite{MS93}.  A big
step towards an understanding of the behaviour in dimension
$d=4-\epsilon$ is taken in \cite{MS08} (their model is formulated on a
lattice in dimension $3$ but it mimics the behaviour of the
nearest-neighbour model in dimension $4-\epsilon$).

Our renormalisation group method is a greatly extended and
generalised form of work in \cite{BEI92,BI03c,BI03d} for the
hierarchical lattice and \cite{BY90,DH92,BDH98,BMS03} for continuum
quantum field theory.  Details will appear in
\cite{BS11}.  Our method is
based on an exact functional integral representation of the two-point
function of the continuous-time self-avoiding walk as the two-point
function of a quantum field theory containing both bosonic and
fermionic fields.  Such representations have been recently summarised
in \cite{BIS09}.

\subsection{Background}

A self-avoiding walk on the simple cubic lattice $\Zd$ is an
\emph{injective} map
\begin{equation}
    \omega:\{0,1,\dots ,n \} \rightarrow \Zd
\end{equation}
such that for all $i$, $\omega (i)$ and $\omega (i+1)$ are nearest
neighbours in $\Zd$. We call $n$ the number of steps. The main result
of this article will actually be a statement about about random maps
$X:[0,T] \rightarrow \Zd$, but to explain the background we start with
self-avoiding walk.

\begin{figure}[h]
\begin{center}
\includegraphics[scale=.33]{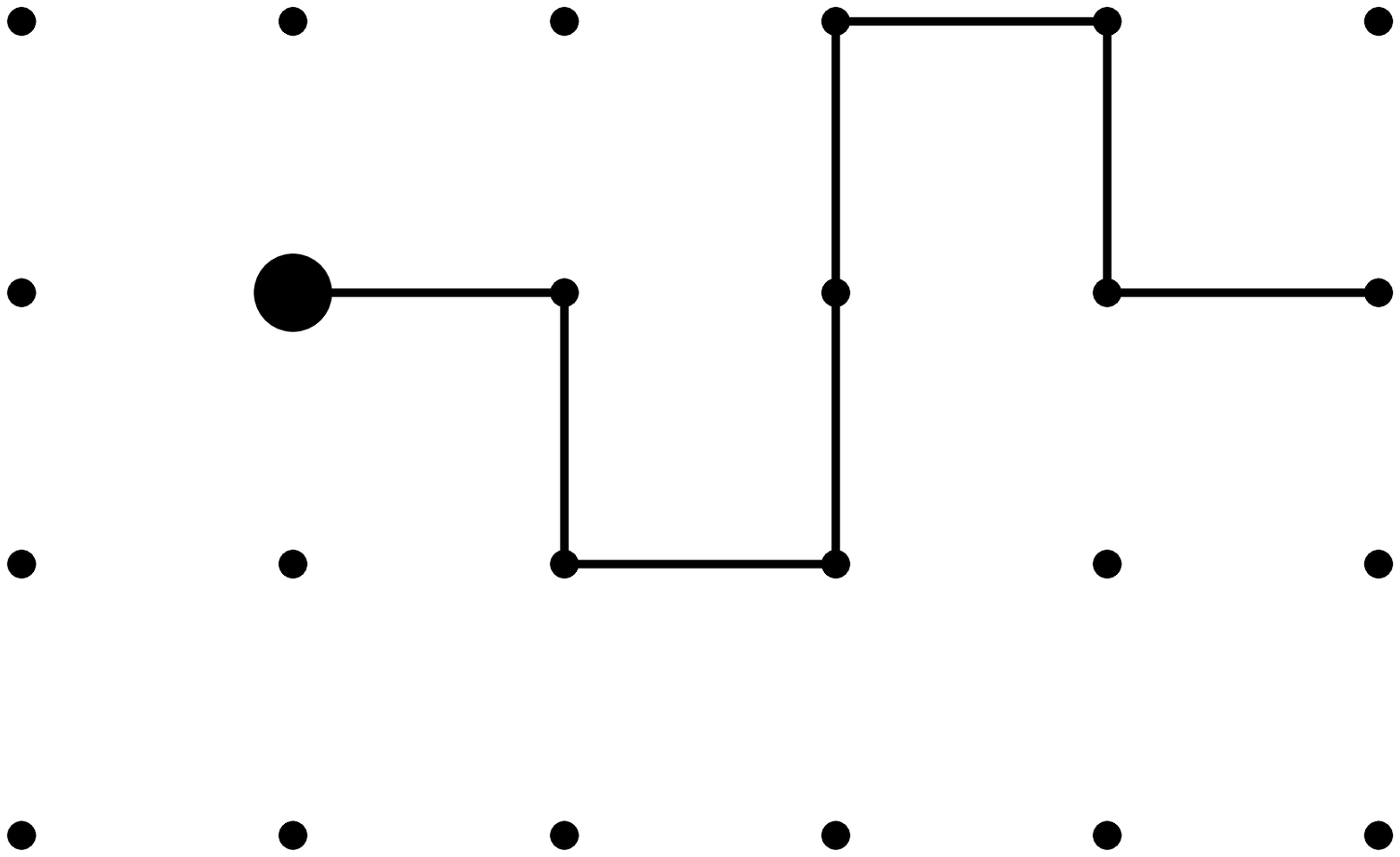}
\end{center}
\caption{\label{fig:8 steps}
An 8 step self-avoiding walk on $\Zd,\,  d=2$.}
\end{figure}

Let ${\cal S}_n$ be the set of all self-avoiding walks with $n$ steps
and with $\omega (0) =0$.  Let $c_{n}$ be the number of elements in
${\cal S}_n$.  By declaring that all $\omega$ in ${\cal S}_n$ have
equal probability $1/c_{n}$ we make ${\cal S}_n$ into a probability
space with expectation $\Ebold_{n}$. The subscript $n$ reminds us that
the probability space depends on $n$. In the sequel ``model'' means a
choice of probability space and law.

This model arose in statistical mechanics.  It is, for example, a
natural model when one is interested in the conformation of linear
polymer molecules.  There is another natural model called the
\emph{true} or \emph{myopic} self-avoiding walk.  Unlike our model,
true self-avoiding walk is a stochastic process which at each step
looks at its neighbours and chooses uniformly from those visited least
often in the past. Recent progress on this model is reported in
\cite{HTV09}.

The key problem is to determine the growth in $n$ of the mean-square
displacement,
\begin{equation}
    \Ebold_{n} |\omega (n)|^{2}
    =
    c_{n}^{-1}\sum_{\omega \in \Scal_{n}} |\omega (n)|^{2}
    ,
\end{equation}
where $|\omega (n)|$ is the Euclidean norm of $\omega (n)$ as an
element of $\Zd$. More precisely, we want to prove the existence of
$\nu$ such that
\begin{equation}
    \lim_{n\rightarrow \infty}n^{-2\nu}\Ebold_{n} |\omega (n)|^{2}
    \in
    (0,\infty)
    ,
\end{equation}
and we want to calculate $\nu$. We will call this the $\nu$ problem.

As explained in \cite[page 16]{MS93}, there is an easier version of
this problem that we will call the \emph{Abelian} $\nu$ problem,
because proving the existence of $\nu $ after solving the Abelian
problem is a Tauberian problem.  Let $\Scal = \bigcup_{n} \Scal_{n}$
and let $n (\omega) =n$ for $\omega \in \Scal_{n}$. For $z >0$ we
define the \emph{two-point function}
\begin{equation}
    G_{z} (x)
    =
    \sum_{\omega \in \Scal} z^{n (\omega)} \1_{\omega (n (\omega))=x}
    .
\end{equation}
Let
\begin{equation}
    \chi^{(p)}
    =
    \sum_{\omega \in \Scal} z^{n (\omega)} |\omega (n (\omega))|^{p}
    =
    \sum_{x \in \Zd} G_{z} (x) |x|^{p}
    .
\end{equation}
The Abelian version of the $\nu$ problem is to determine the growth of
$\sqrt{\chi^{(2)}}/\chi^{(0)}$ as $z\uparrow z_{c}$, where $z_{c}$ is
the common radius of convergence of the power series in this ratio. If
$\nu$ exists then it equals the Abelian $\nu$. In dimensions $d\ge 5$,
according to the following theorem, $\nu =1/2$.

\begin{theorem}
\label{thm:d5}
\cite{HS92b,HS92a}
For $d \geq 5$, there are positive constants $A,D,c,\epsilon$ such that
\begin{align}
    c_n & =   A\mu^n [ 1+O(n^{-\epsilon})], \\
    \Ebold_{n} |\omega(n)|^2
    & =  Dn[ 1+O(n^{-\epsilon})],
\end{align}
and the rescaled self-avoiding walk converges weakly to Brownian motion:
\begin{equation}
    \label{e:scaling-limit}
    \frac{\omega (\lfloor nt \rfloor)}{\sqrt{Dn}}
    \Rightarrow  B_t.
\end{equation}
Also \cite{Hara08},
as $|x|\to\infty$,
\begin{equation}
    \label{e:critical-G}
    G_{z_c}(x) = c |x|^{-(d-2)}[1+O(|x|^{-\epsilon})].
\end{equation}
\end{theorem}

The limit in \eqref{e:scaling-limit} is called a scaling limit.  The
identification of scaling limits for dimensions $d=2,3,4$ is the grand
goal, but the $\nu$ problem is a key intermediate objective because
$n^{-\nu}\omega (\lfloor nt \rfloor)$ is the candidate sequence for
the scaling limit.

If we set up the probability space without imposing the injective
condition in the definition of $\omega $, then the mean-square
displacement is exactly $n$, because then the law for $\omega$ is that
of simple random walk. According to Donsker's Theorem, the scaling
limit of simple random walk, with $D=1$, is also Brownian motion.
Thus, in dimensions $d \ge 5$ self-avoiding walk and simple random
walk have the same scaling limit. When different models have the same
scaling limit, we say the models are in the same \emph{universality
class}. One of the goals of mathematical statistical mechanics is to
classify universality classes. 

Theorem~\ref{thm:d5} will not hold with $\nu =1/2$ for dimensions four
and less. There is a conjecture going back to \cite{BGZ73} that, for
$d=4$,
\begin{equation}
\label{e:d4logs}
    c_n  \sim A \mu^n (\log n)^{1/4}, \quad
    \Ebold_{n} |\omega (n)|^2  \sim D n (\log n)^{1/4}
    .
\end{equation}
This and the next paragraph motivates our interest in four dimensions.

In dimension $d=3$, nothing is known rigorously about the $\nu$
problem. The existence of $\nu$ is not proved.  It is not known that
self-avoiding walk moves away from the origin faster than simple
random walk, $\Ebold_{n} |\omega(n)|^2 \ge n$, nor is it known that
self-avoiding walk is slower than ballistic, $\Ebold_{n} |\omega(n)/n|^2
\rightarrow 0$.  In dimension $d=2$, there is the same basic lack of
control as in $d=3$, but the good news is that there is a candidate
for the scaling limit, which tells us that if $\nu$ exists it should
be equal to $3/4$.  In \cite{LSW04}, the process known as ${\rm
SLE}_{8/3}$ is identified as the scaling limit of self-avoiding walk
subject to the unproven hypothesis that the scaling limit exists and
is conformally invariant.

SLE is a breakthrough discovery because it is provides a comprehensive
list of possible scaling limits in $d=2$.  It has separated off the
issues of existence of limits and universality and made it possible to
study candidate limits without first having proved they are limits.
On the other hand, theoretical physicists have a profound calculus
called the \emph{Renormalisation Group} (RG) that naturally explains
when different models are in the same universality class and that can
also prove the existence of limits.  We will follow this path. RG, in
the form that we will develop, was largely invented by Ken Wilson
\cite{Wil71,WiKo74,Wil83}.  RG as a rigorous tool originated with
\cite{BCG+80,GaKu83}.  Later developments are reviewed
in \cite{Bryd09}.  The hierarchical lattices mentioned earlier have
special properties that greatly simplify RG.  The $n (\log n)^{1/4}$
growth of \eqref{e:d4logs} has been shown to hold for continuous-time
weakly self-avoiding walk on a four dimensional hierarchical lattice
in \cite{BEI92,BI03c,BI03d}. Very recently, the corresponding Abelian
$\nu$ problem has been solved in \cite{HO10} for a
\emph{discrete-time} model on the hierarchical lattice.

\subsection{Continuous-time weakly self-avoiding walk and the main
result} \label{sec:main1}
We now describe a probability law on a space of maps $X:[0,T]
\rightarrow \Zd$.  We use the word ``time'' for the parameter $t \in
[0,T]$, but as for the discrete-time case there is a different space
and law for each $T$.  It is not a stochastic process which reveals
more about itself as ``time'' advances, so it is better to think of the
interval $[0,T]$ as a continuous parametrisation of a path in $\Zd$.

Fix a dimension $d \ge 4$.  Let $X$ be the continuous-time simple
random walk on $\Zd$ with ${\rm Exp}(1)$ holding times and
right-continuous sample paths.  In other words, the walk takes its
nearest neighbour steps at the events of a rate-1 Poisson process. Let
$P_{a}$ and $E_a$ be the probability law and the expectation for this
process started in $X(0)=a$.  The local time at $x$ up to time $T$ is
given by
\begin{equation}
    L_{x,T} = \int_0^T \1_{X(s)=x}\,ds
    ,
\end{equation}
and we can measure the amount of self-intersection experienced by
$X$ up to time $T$ by
\begin{align}
    I (0,T)
    &=
    \int_0^T ds_1  \int_0^T ds_2
    \1_{X(s_1)= X(s_2)}
    \nnb
    &=
    \int_0^T ds_1  \int_0^T ds_2
    \sum_{x\in\Zd} \1_{X(s_1)=x} \1_{X(s_2)=x}
    =
    \sum_{x\in\Zd} L_{x,T}^2
    .
\end{align}
Then, for $g>0$, $e^{-g I (0,T)}$ is our substitute for the indicator
function supported on self-avoiding $X$. For $g>0$, we define a new
probability law
\begin{equation}
    P_{g,a} (A)
    =
    E_{a} (e^{-g I (0,T)}\1_{A})
    /
    E_{a} (e^{-g I (0,T)})
\end{equation}
on measurable subsets $A$ of the set of all maps $X:[0,T]\rightarrow
\Zd$ with $X (0)=a$. For this model there is a $\nu$
problem\footnote{solved on the hierarchical lattice for $g$ small in
\cite{BEI92,BI03c,BI03d}}, but only the Abelian $\nu$ problem for
$\Zbold^{d}$ is currently within the reach of the methods of this
paper.

The continuous-time weakly self-avoiding walk two-point function is
defined by
\begin{equation}
\label{e:Gwsaw}
    G_{g,\nu}(a,b)
    =
    \int_0^\infty
    E_{a} (e^{-gI (0,T)}\1_{X(T)=b})
    e^{- \nu T}
    dT,
\end{equation}
where $\nu$ is a parameter (possibly negative) which is chosen in such
a way that the integral converges.  For $p \ge 0$ define
\begin{equation}
    \label{e:chi}
    \chi_{g}^{(p)}(\nu) =   \sum_{b\in \Z^d}  G_{g,\nu}(a,b) |b-a|^{p}.
\end{equation}
By subadditivity, c.f. \cite{MS93}, there exists $\nu_{c}=\nu_{c}(g)$ such
that $\chi_g^{(0)}(\nu) < \infty$ if and only if $\nu > \nu_{c}$.  We
call this $\nu_{c} $ the \emph{critical value of $\nu$}. Our main result
is the following theorem.

\begin{theorem}
\label{thm:wsaw4} Let $d \ge 4$.  There exists $g_{\mathrm{max}}>0$
such that for each $g \in [0,g_{\mathrm{max}}]$ there exists $c_g>0$
such that as $|\pp-\qq|\to \infty$,
\eq
\lbeq{Gasy}
    G_{g,\nu_{c}(g)}(a,b) = \frac{c_g}{|\pp-\qq|^{d-2}}\left( 1 + o(1)\right).
\en
\end{theorem}

This is the analogue of \eqref{e:critical-G} in Theorem~\ref{thm:d5},
but now including dimension $d=4$.  There are no log corrections.  Log
corrections are only expected in the singular behaviour of
$\chi_{g}^{(p)}(\nu)$ as $\nu \downarrow \nu_{c}$ for $p\ge 0$. The
case $g=0$ is a standard fact about simple random walk; our proof
is given for case $g>0$.

\section{Finite volume approximation}
\label{sec:fv}

In this section we describe the first step in our proof, which is to
approximate the \emph{infinite volume} $\Zd$ by \emph{finite volume},
namely a discrete torus.

We do not make explicit the dependence on $g$, which is fixed and
positive. Let $R \ge 3$ be an integer, and let $\Lambda = \Z^d/ R\Z^d$
denote the discrete torus of side $R$.  For $a,b\in \Lambda$, let
\begin{equation}
\label{e:GwsawLambda}
    G_{\Lambda,\nu}(a,b)
    =
    \int_0^\infty
    E_{a,\Lambda} \left(
    e^{-gI(0,T)}
    \1_{X(T)=b} \right)
    e^{- \nu T}
    dT,
\end{equation}
where $E_{a,\Lambda}$ denotes the continuous-time simple random walk
on $\Lambda$, started from $a$.  The following theorem shows that it
is possible to study the critical two-point function in the double
limit in which first $\Lambda \uparrow \Zd$ and then $\nu \downarrow
\nu_{c}$.  We will follow this route, focusing our analysis on the
subcritical finite volume model with sufficient uniformity to take the
limits.

\begin{theorem}
\label{thm:Glims}
Let $ d \ge 1$ and $\nu \ge  \nu_{c}$.  Then
\eq
\lbeq{Givlc}
    G_\nu(a,b)
    =
    \lim_{\nu' \downarrow \nu} \lim_{\Lambda \uparrow \Zd}
    G_{\Lambda,\nu'}(a,b).
\en
\end{theorem}

\section{Integral representation}
\label{sec:intrep}

The next step in the proof is to represent the two-point function in
finite volume by an integral that we will approximate by a Gaussian
integral.

Recall that $\Lambda$ denotes a discrete torus in $\Zd$.  Given
$\varphi \in \Cbold^{\Lambda}$ and writing $\varphi = (\varphi_x),\,x
\in \Lambda$, we write $d\varphi_x$ and $d\bar\varphi_x$ for the
differentials, we fix a choice of the square root $\sqrt{2\pi i}$, and
we set
\begin{equation}
    \psi_x = \frac{1}{\sqrt{2\pi i}} d\varphi_x,
    \quad
    \bar\psi_x = \frac{1}{\sqrt{2\pi i}} d\bar\varphi_x.
\end{equation}
Define the differential forms
\begin{equation}
    \tau_x
    = \varphi_x \bar\varphi_x
    + \psi_x  \wedge \bar\psi_x
    \quad\quad (x \in \Lambda),
\end{equation}
and
\begin{equation}
\label{e:addDelta}
    \tau_{\Delta,x}
    =
    \frac 12 \Big(
    \varphi_{x} (- \Delta \bar{\varphi})_{x} + (- \Delta \varphi)_{x} \bar{\varphi}_{x} +
    \psi_{x} \wedge (- \Delta \bar{\psi})_{x} + (- \Delta \psi)_{x} \wedge \bar{\psi}_{x}
    \Big),
\end{equation}
where $\Delta$ is the lattice Laplacian on $\Lambda$ defined by
$\Delta \varphi_{x} = \sum_{y: |y-x|=1} (\varphi_{y} - \varphi_{x})$,
and $\wedge$ is the standard wedge product.
From now on, for differential forms $u,v$, we will abbreviate by
writing $uv = u\wedge v$.  In particular $\psi_{x} \psi_{y} = -
\psi_{y} \psi_{x}$ and likewise $\psib_{x}$ anticommutes with
$\psib_{y}$ and with $\psi_{y}$.  The proof of the following
proposition is given in \cite{BEI92,BI03d}; see also \cite{BIS09} for
a self-contained proof.

\begin{prop}
\label{prop:Grep} Given $g>0$, let $\nu$ be such that
$G_{\Lambda,\nu}(a,b)$ is finite. Then
\begin{equation}
\label{e:Grep1}
    G_{\Lambda,\nu}(a,b)
    =
    \int_{\Cbold^{\Lambda}}
    e^{-\sum_{x\in\Lambda}(\tau_{\Delta ,x} + g \tau_x^2 + \nu \tau_x)}
    \bar\varphi_{\pp} \varphi_{\qq}
    .
\end{equation}
\end{prop}

The definition of an integral such as the right-hand side of
\refeq{Grep1} is as follows:
\begin{enumerate}
\item Expand the entire integrand in a power series about its
degree-zero part (this is a \emph{finite} sum due to the
anti-commutativity of the wedge product, and the order of factors in
the resulting products is immaterial due to the even degree), e.g.,
\begin{equation}
    e^{-\nu \tau_{x} }
    = e^{-\nu \varphi_{x}\varphib_{x} - \frac{1}{2\pi i} d\varphi_{x} d\bar\varphi_{x}}
    =  e^{-\nu \varphi_{x}\varphib_{x}}
    \left(1 - \frac{1}{2\pi i} d\varphi_{x} d\bar\varphi_{x}\right).
\end{equation}
In general, any function of the differentials is defined by its formal
power series about its degree-zero part.
\item Keep only terms with one factor $\psi_x$ and one $\bar\psi_x$
for each $x\in\Lambda$, write $\varphi_x = u_x + i v_x$,
$\bar\varphi_x = u_x - i v_x$ and similarly for the differentials.
\item Rearrange the differentials to $\prod_{x\in\Lambda} du_x dv_x$,
using the anti-commutativity of the wedge product.
\item Finally, perform the Lebesgue integral over $\R^{2|\Lambda|}$.
\end{enumerate}
This is explained in more detail in \cite{BIS09}.  These integrals
have the remarkable self-normalisation property that
\begin{equation}
\label{self-norm}
    \int
    e^{
    -\sum_{x\in\Lambda}(a_{x}\tau_{\Delta ,x} + b_{x} \tau_x^2 + c_{x} \tau_x)
    }
    =
    1,
    \quad\quad
    a_{x}\ge 0,\, b_{x}> 0,\, c_{x} \in \Rbold,\, x\in \Lambda
    .
\end{equation}
Self-contained proofs of this, and of generalisations, can be found in
\cite{BIS09}.  The variables $\varphi_{x}$ and the forms $\psi_{x}$
are called \emph{fields}.

\section{Quadratic or Gaussian approximation}

The integral representation of Proposition~\ref{prop:Grep} opens a
natural route for approximation by non-interacting walk with different
parameters. To do this we split the exponent $\tau_{\Delta ,x} + g
\tau_x^2 + \nu \tau_x$ in \eqref{e:Grep1} into a part which is
quadratic in the variables $\varphi$ and a remainder.  When the
remainder is ignored the rest of the integral becomes Gaussian and the
Gaussian integral represents a non-interacting walk.  It is important
not to assume that the best approximation is the quadratic terms
$\tau_{\Delta ,x} + \nu \tau_x$.  We even want to allow
$\tau_{\Delta}$ to be divided up.  To see what a different coefficient
in front of $\tau_{\Delta}$ means we make the change of variable
$\varphi_{x} \mapsto \sqrt{1+z_{0}}\varphi_{x}$, with  $z_{0} >
-1$.  This gives
\begin{equation}
    G_{\Lambda,\nu}(a,b)
    =
    (1+z_{0})\int_{\Cbold^{\Lambda}}
    e^{-\sum_{x\in\Lambda}\big(
    (1+z_{0})\tau_{\Delta ,x} + g (1+z_{0})^{2} \tau_x^2 + \nu (1+z_{0})\tau_x
    \big)}
    \bar\varphi_{\pp} \varphi_{\qq}
    ,
\end{equation}
where the Jacobian is contained in the transformation of $\psi ,
\psib$. Then, for any $m^{2}\ge 0$, simple algebra allows us to
rewrite this as
\begin{equation}
    \label{e:z-indep}
    G_{\Lambda,\nu}(a,b)
    =
    (1+z_{0})
    \int
    e^{-S (\Lambda) - \tilde{V}_{0} (\Lambda)
    }
    \bar\varphi_{\pp} \varphi_{\qq}
    ,
\end{equation}
where
\begin{align}
    \label{e:Sdef}
    &
    S (\Lambda)
    =
    \sum_{x\in\Lambda}
    \big(\tau_{\Delta ,x} + m^{2} \tau_x \big)
    ,
    \\
    \label{e:Vtil0def}
    &
    \tilde{V}_{0} (\Lambda)
    =
    \sum_{x\in\Lambda}
    \big(g_{0} \tau_x^2 + \nu_{0} \tau_x + z_{0}\tau_{\Delta ,x}\big)
    ,
    \\
    \label{e:zmdef}
    &
    g_{0} = (1+z_{0})^{2} g,
    \quad \quad
    \nu_{0} = (1+z_{0}) \nu_{c},
    \quad \quad
    m^{2} = (1+z_{0}) (\nu - \nu_{c})
    ,
\end{align}
and $\nu_{c}$ was defined below \eqref{e:chi}. The two-point function
$G_{\Lambda,\nu}(a,b)$ in \eqref{e:z-indep} does not depend on
$(z_{0},m^{2})$ so, in the next theorem, these are free parameters
that do not get fixed until Section~\ref{sec:focc}.  In view of
Theorem~\ref{thm:Glims} and Proposition~\ref{prop:Grep}, to prove
Theorem~\ref{thm:wsaw4} it suffices to prove the following theorem.

\begin{theorem}
\label{thm:wsaw4rep} Let $d\ge 4$.  There exists $g_{\mathrm{max}}>0$
such that for each $g \in [0,g_{\mathrm{max}}]$ there exist $c
(g) >0$ such that as $|\pp-\qq|\to \infty$,
\begin{equation}
\lbeq{eta0}
    \lim_{\nu \downarrow \nu_{c}}
    \lim_{\Lambda \uparrow \Zd}
    (1+z_{0})
    \int_{\Cbold^{\Lambda}}
    e^{-S (\Lambda) - \tilde{V}_{0} (\Lambda)}
    \bar\varphi_{\pp} \varphi_{\qq}
    =
    \frac{c (g)}{|\pp-\qq|^{d-2}}\left( 1 + o(1)\right).
\end{equation}
\end{theorem}

To prove Theorem~\ref{thm:wsaw4rep}, we study the integral on the
left-hand side via a renormalisation group analysis, without making
further direct reference to its connection with self-avoiding walks.
In order to calculate this integral we define, for $\sigma \in
\Cbold$,
\begin{equation}
    \label{e:V0def}
    V_{0} (\Lambda )
    =
    \tilde{V}_{0} (\Lambda)
    +
    \sigma \varphib_{\pp} + \sigmab \varphi_{\qq}
\end{equation}
and use
\begin{equation}
    \label{e:generating-fn}
    \int_{\Cbold^{\Lambda}}
    e^{-S (\Lambda) - \tilde{V}_{0} (\Lambda)}
    \bar\varphi_{\pp} \varphi_{\qq}
    =
    -\left.
    \frac{\partial }{\partial \sigma }
    \frac{\partial }{\partial \sigmab }
    \right|_{0}
    \int_{\Cbold^{\Lambda}}
    e^{-S (\Lambda) - V_{0} (\Lambda)}
    .
\end{equation}
We will call $\sigma$ an \emph{external field}.

\section{Forms and test functions}

In this section we introduce notation for handling the differential
forms that appear in Theorem~\ref{thm:wsaw4rep}. We will write
\emph{form} in place of ``differential forms'' from now on.  We focus
on dimension $d = 4$, but leave $d$ in various formulas since $4$ can
also appear for other reasons.

\subsection{The space $\Ncal$}

A form is a polynomial in $\psi ,\psib$
with coefficients that are functions of $(\varphi,\sigma ) \in
\Cbold^{\Lambda}\times \Cbold$.

Given $\sigma \in \Cbold$ we define $\sigma_{1}=\sigma$ and
$\sigma_{2}=\sigmab$ so that $\sigma$ can be identified with a
function $\sigma :\{1,2 \}\rightarrow \Cbold$.  Similarly, let
$\Lambda_2 = \Lambda \times \{1,2\}$ so that given $\varphi \in
\Cbold^{\Lambda}$ we have the function on $x = (s,i) \in \Lambda_{2}$
defined by
\begin{equation}
    \phi_{x}
    =
    \begin{cases}
    \varphi_{s}&i=1,\\
    \varphib_{s}&i=2.
    \end{cases}
\end{equation}
Since $\phi$ and $\varphi$ are in one to one correspondence and since
we are only interested in functions on $\Lambda_{2}$ that arise from
some $\varphi$ we write $\phi \in \Cbold^{\Lambda}$.

Forms are elements of the algebra $\Ncal$ whose generators are the
degree one forms $(\psi_{x} ,\psib_{x},\, x \in \Lambda)$ subject to
the relations that all generators mutually anticommute.  For $x =
(s,i) \in \Lambda_{2}$, we write
\begin{equation}
    \psi_{x}
    =
    \begin{cases}
    \psi_{s}&i=1,\\
    \psib_{s}&i=2.
    \end{cases}
\end{equation}
Then we introduce the space $\Lambda^{*} = \cup_{q=0}^{\infty}
\Lambda_{2}^{q}$ of all sequences in $\Lambda_{2}$ with finitely many
terms so that every monomial in $\psi$ can be written in the form, for
some $y\in \Lambda^{*}$,
\begin{equation}
    \psi^y = \begin{cases}
    1 & \text{if $q=0$}
    \\
    \psi_{y_1}\cdots \psi_{y_q} & \text{if $q \geq 1$.}
    \end{cases}
\end{equation}
The $q=0$ term in $\Lambda^{*}$ is a set consisting of a single
element called the ``empty sequence'', which by definition has length
zero.  Given a sequence $y \in \Lambda^{*}$, $q = q(y)$ is the
length of the sequence and $y!  = q (y)!$.  Every element of $\Ncal$
has the form
\begin{equation}
    \label{e:K}
    F
    =
    F (\phi ,\sigma)
    =
    \sum_{y \in \Lambda^{*}} \frac{1}{y!}F_y (\phi ,\sigma) \psi^y
    .
\end{equation}
Given $x=(x_1,\ldots, x_p)\in \Lambda_2^p$ and $z=(z_1,\ldots,z_r) \in
\{1,2\}^r$, we write
\begin{equation}
    F_{x,y,z}(\phi,\sigma )
    = \frac{\partial^p }{\partial \phi_{x_p} \cdots  \partial \phi_{x_1}}
    \frac{\partial^r }{\partial \sigma_{z_r} \cdots  \partial \sigma_{z_1}}
    F_y(\phi,\sigma).
\end{equation}
For $X \subset \Lambda$, we define $\Ncal (X)$, which is a subspace of
$\Ncal$, by
\begin{equation}
\label{e:9NXdef}
    \Ncal (X) = \{ F \in \Ncal : F_{x,y}=0 \; \text{if any component of $x,y$
    is not in $X$}\}.
\end{equation}
For example $\tau_{x} \in \Ncal (\{x \})$ and $\tau_{\Delta ,x} \in
\Ncal (X)$ where $X=\{y:|y-x|\le 1 \}$.

By introducing
\begin{equation}
    \phi^y = \begin{cases}
    1 & \text{if $q=0$}
    \\
    \phi_{y_1}\cdots \phi_{y_q} & \text{if $q \geq 1$.}
    \end{cases}
\end{equation}
we write the formal Taylor expansion of $F(\phi + \xi)$ in powers
of $\xi$ and $\sigma$ as
\begin{equation}
\label{e:fTexp}
    \sum_{x,y \in \Lambda^*, z\in \{1,2\}^*}
    \frac{1}{x!y!z!}
    F_{x,y,z} (\phi,0) \xi^{x} \psi^y \sigma^z.
\end{equation}
Functions $f:\Lambda^* \times \Lambda^* \times \{1,2\}^* \rightarrow
\Cbold $ are called \emph{test functions}.  We define a pairing
between elements of $\Ncal$ and the set of test functions as follows:
for a test function $f$, for $\phi \in \Cbold^{\Lambda}$, let
\begin{equation}
    \label{e:pairing}
    \pair{F, f}_\phi
    = \sum_{x,y \in \Lambda^*, z\in \{1,2\}^*}
    \frac{1}{x!y!z!}F_{x,y,z}(\phi,0) f_{x,y,z}.
\end{equation}

\subsection{Local polynomials and localisation}
\label{sec:lml}

For a function $f:\Lambda \rightarrow \Cbold$ and $e$ a unit vector in
$\Zd$ we define the \emph{finite difference derivative} $(\nabla_{e}
f)_{x} = f (x+e) - f (x)$.  Repeated differences such as
$(\nabla_{e}\nabla_{e'} f)_{x}$ are called \emph{derivatives}.

A \emph{local monomial} is a product of finitely many fields and
derivatives of fields such as $M = \psi \psib \nabla_{e}\varphib$.
Using this example to introduce a general notation, given $x \in
\Lambda$ let $M_{x}= \psi_{x} \psib_{x} (\nabla_{e}\varphib)_{x}$,
and given $X\subset \Lambda$ let $M (X) = \sum_{x\in X}M_{x}$.
\emph{Local polynomials} are finite sums of local monomials with
constant coefficients.

An important example of a local polynomial is
\begin{equation}
    \label{e:Vdef}
    V = g\tau^2 + \nu\tau + z\tau_{\Delta,x} +
    \lambda \1_{a}\sigmab \varphi + \lambda \1_{b} \sigma \varphib + q\sigmab \sigma
    ,
\end{equation}
which extends the local polynomial of \refeq{V0def} by the addition of
the term $q\sigmab \sigma$.
The indicator function
$\1_{a}:\Lambda \rightarrow \{0,1 \}$ equals $1$ when evaluated
on $a$ and is zero otherwise. The parameters $(g,\nu ,z,\lambda,q)$
are called \emph{coupling constants}.

\emph{Euclidean symmetry:} The lattice $\Zd$ has automorphisms $E:\Zd
\rightarrow \Zd$.  An example for $d=1$ is $Ex = 1 - x$.  By letting
an automorphism $E$ act on the spatial labels on fields,
$\varphi_{x} \mapsto \varphi_{Ex}$, $E$ induces an action, $E:\Ncal
\rightarrow \Ncal$.  A local polynomial $P$ is \emph{Euclidean
invariant} if automorphisms of $\Zd$ that fix $x$ also fix $P_{x}$.
For example, $\psi \psib \nabla_{e}\varphib$ is not Euclidean
invariant because there is a reflection that changes $\varphi_{x+e}$
into $\varphi_{x-e}$ so that $(\nabla_{e}\varphib)_{x} \mapsto
(\nabla_{-e}\varphib)_{x}$.  On the other hand, the term
$\tau_{\Delta}$ in \eqref{e:Vdef} is a
Euclidean invariant local monomial.

\emph{Gauge invariance:} A local polynomial is gauge invariant if it
is invariant under the \emph{gauge flow}: $(\sigmaa,\varphi )
\rightarrow (e^{i\theta}\sigmaa, e^{i\theta}\varphi)$. Thus $V$ of
\eqref{e:Vdef} is gauge invariant.

\emph{Supersymmetry:} There is an antiderivation $\hat{Q}:\Ncal
\rightarrow \Ncal$ characterised by
\begin{align}
    &
    \hat{Q}\varphi_{x} = \psi_{x},
    \quad \quad
    \hat{Q}\psi_{x} = - \varphi_{x},
    &
    \hat{Q}\varphib_{x} = \psib_{x},
    \quad \quad
    \hat{Q}\psib_{x} = \varphib_{x}
    .
\end{align}
An element of $F \in\Ncal$ is said to be \emph{supersymmetric} if
$\hat{Q}F=0$.  The terms $\tau , \tau_{\Delta}, \tau^{2}$ in $V$ are
supersymmetric local monomials. The forms $\sigmab \varphi, \sigma
\varphib, \sigmab \sigma$ are gauge invariant, but not
supersymmetric. It is straightforward to check that $\hat{Q}^{2}$
generates the gauge flow.  Therefore supersymmetry implies gauge
invariance.  Further details can be found in \cite{BIS09}.

The pairing \eqref{e:pairing} defines $F \in \Ncal$ as a linear
function, $f \mapsto \pair{F,f}_{0}$, on test functions. The
subscript means that we set $\phi =0$. Let $\Phipol$ be a set of test
functions. Two elements $F_{1}$ and $F_{2}$ of $\Ncal$ are equivalent
when they define the same linear function on $\Phipol$. We say they
are \emph{separated} if they are not equivalent.

\begin{example}
Let $\Phipol$ be the set of test functions that are linear in their
$\Lambda$ arguments. Fix a point $k \in \Zd$. Let $F = \varphi_{k}$,
and let $F' = \varphi_{0} + (k \cdot \nabla \varphi)_{0}$. Then $F$
and $F'$ are equivalent because a linear test function $f (x) = a
+b\cdot x$ cannot separate them:
\begin{equation}
    \pair{F,f}
    =
    a + b\cdot k,
    \quad \quad
    \pair{F',f}
    =
    a + k \cdot \nabla f
    =
    a + b\cdot k.
\end{equation}
To avoid confusion let us emphasise that two different contexts for
``polynomial'' are in use: a test functions can be a polynomial in
$x\in \Lambda$, while local polynomials are polynomial in fields.
\end{example}

The choice for $\Phipol$ in this example is not the one we want. The
details in the definition given below are less important than the
objective of the definition, which is that $\Phipol$ should be a
minimal space of test functions that separates the terms in \eqref{e:Vdef}.

We define $\Phipol$ to be the set of test functions $f (x,y,z)$ that
are polynomial in the $\Lambda$ arguments of $(x,y) \in \Lambda^*
\times \Lambda^*$ with restrictions on degree listed below.  For $f
\in \Phipol$, as a polynomial in the $x,y$ components in $\Lambda$:
\begin{enumerate}
\item The restriction of $f$ to $(x,y,z)$ with $r(z)=0$ has total
degree at most $d-p (x)[\phi]-q (y)[\phi]$;  $f (x,y,z)=0$ when $d-p
(x)[\phi] - q (y)[\phi] < 0$. Here
\begin{equation}\label{e:dimphidef}
    [\phi] = (d-2)/2
    .
\end{equation}
For dimension $d=4$, $[\phi] = 1$.
\item The restriction of $f$ to $(x,y,z)$ with $r(z)=r \in\{1,2\}$ has
total degree at most $r - p (x) - q(y)$; $f (x,y,z)=0$ if $r-p (x) -
q(y) < 0$ or $r>2$.
\end{enumerate}
Let $\Vcal$ be the vector space of gauge invariant local polynomials
that are separated by $\Phipol$ and, for $X\subset \Lambda$, let
$\Vcal (X) = \{P (X):P\in \Vcal \}$. The following proposition
associates to any form $F \in \Ncal$ an equivalent local polynomial in
$\Vcal (X)$ \cite{BS11}.

\begin{prop}
\label{prop:9LTdef} For $X \subset \Zd $ there exists a linear map
$\LTbar_{X}:\Ncal \rightarrow \Vcal (X)$ such that
\begin{align}
    &(a)\quad
    \pair{\LTbar_{X} F, f}_0 = \pair{F,f}_0
    \quad
    \text{for $f \in \Phipol$, $F \in \Ncal$},
    \\
    &(b)\quad
    E\big(\LTbar_{X} F\big) = \LTbar_{EX} (EF) \quad
    \text{for automorphisms $E:\Zd \rightarrow \Zd$, $F \in \Ncal$},
    \\
    &(c)\quad
    \LTbar_{X'}\circ \LTbar_X = \LTbar_{X'} \quad
    \text{for $X,X' \subset \Lambda$}
    .
\end{align}
\end{prop}

Let $\Vcal_{H} \subset \Vcal$ be the subspace generated by monomials
that are not divisible by $\sigma$ or $\sigmab$, and let $\Vcal_{O}
\subset \Vcal$ be the subspace generated by monomials that are
divisible by $\sigma$ or $\sigmab$. Then $\Vcal = \Vcal_{H} \oplus
\Vcal_{O}$, and on this direct sum we define
\begin{equation}
    \label{e:LTdef}
    \LT_{X} = \LTbar_{X}\oplus \LTbar_{X\cap \{a,b \}}
    .
\end{equation}
Symmetry considerations for the integral representation restrict the
domain of $\LTbar$ in our applications so that its range reduces to
polynomials of the form $V$ as in \eqref{e:Vdef}.

\section{Gaussian integration}

\subsection{The super-expectation}
\label{sec:gaussian-integration}

For a $\Lambda \times \Lambda$ matrix $A$, we define
\begin{equation}
    S_{A} (\Lambda )
    =
    \sum_{x,y \in \Lambda}\Big(
    \varphi_{x} A_{xy} \varphib_{x} +
    \psi_{x} A_{xy} \psib_{y}
    \Big)
    .
\end{equation}
When $A = m^{2}-\Delta$ this is the same as $S (\Lambda)$ which was
defined in \eqref{e:Sdef}.  Let $C$ be a positive-definite $\Lambda
\times \Lambda$ matrix.  Then $A=C^{-1}$ exists. We introduce the
notation
\begin{equation}
    \Ebold_{C} F
    =
    \int_{\Cbold^{\Lambda}}
    e^{-S_A (\Lambda)} F,
\end{equation}
for $F$ a form in $\Ncal$.  The integral is defined as described under
Proposition~\ref{prop:Grep}.  We call $C$ the covariance because
$\Ebold_{C} \phib_{a} \phi_{b} = C_{ab}$. More generally, if $F$ is a
form of degree zero, i.e., a function of $\phi$, then $\Ebold_{C}F $ is
a standard Gaussian expectation for a complex valued random variable
$\phi$ with covariance $C$ \cite{BIS09}.

We define a space $\Ncal^\times$ in the same way as $\Ncal$ is
defined, but with $\phi$ doubled to $(\phi ,\xi)$ so that
$(\phi,\psi)$ doubles to the pair $(\phi,\psi), (\xi, \eta)$ with
$\eta = (2\pi i)^{-1/2} d \xi$.  The external field $\sigma $ is not
doubled.  We define $\theta: \Ncal \to \Ncal^\times$ by
\begin{equation}
\label{e:taudef}
    (\theta F)(\phi, \xi) =
    \sum_{y\in \Lambda^*} \frac{1}{y!} F_y(\phi+\xi) (\psi + \eta)^y.
\end{equation}
We write $\Ebold_{C} \theta F$ for the element of $\Ncal$ obtained
when the integral over $\Cbold^{\Lambda}$ in $\Ebold_{C}$ applies
\emph{only} to $(\xi,\eta)$.  In the general case where $F$ is a form
this is not standard probability theory, because $\Ebold_{C} \theta F$
takes values in $\Ncal$. To keep this in mind we call this a
\emph{super-expectation}.  The variables and forms $(\xi, \eta)$ that
are integrated out are called \emph{fluctuation fields}.

\subsection{Finite-range decomposition of covariance}
\label{sec:decomposition}

Suppose $C$ and $C_{j},\,j=1,\dots ,N,$ are positive-definite $\Lambda
\times \Lambda$ matrices such that
\begin{equation}
    C = \sum_{j=1}^N C_j
    .
\end{equation}
Let $C' = \sum_{k=2}^N C_k$.  Then, as in the standard theory of
Gaussian random variables, the $\Ebold_{C}$ expectation can be
performed progressively:
\begin{equation}
    \label{e:progressive}
    \Ebold_{C}F
    =
    \Ebold_{C'+C_{1}}F
    =
    \Ebold_{C'}\big(\Ebold_{C_{1}}\theta F\big)
    .
\end{equation}
For further details, see \cite{BS11}. 

From now on we work with $C = (m^{2} - \Delta)^{-1}$, where $\Delta$
is the finite difference Laplacian on the periodic lattice $\Lambda$.
Given any sufficiently large dyadic integer $L$, there exists a
decomposition $C =\sum_{j=1}^{N} C_{j}$ such that $C_{j}$ is
positive-definite and
\begin{equation}
    C_j(x,y) = 0 \quad \text{if} \quad \text{$|x-y| \ge L^j$}.
\end{equation}
This is called the \emph{finite range} property.  The existence of
such a decomposition is established in \cite{BGM04} for the case where
$\Lambda$ is replaced by $\Zd$.  In \cite[Lecture 2]{Bryd09} it is
briefly explained how the decomposition for the periodic $\Lambda$
case is obtained from the $\Zd$ case, for $\Lambda$ a torus of side
$L^{N}$.  To accommodate this restriction on the side of $\Lambda$ the
infinite volume limit in Theorem~\ref{thm:wsaw4rep} is taken with a
sequence of tori with sides $L^{N},\, N\in \Nbold$.

We conclude this section with an informal discussion of scaling
estimates that guide the proof.  Equation \eqref{e:progressive} says
that $F$, which depends on a field with covariance $C$, can be
replaced by $\Ebold_{C_{1}}\theta F$, which depends on a field
characterised by the covariance $C'$. Repeating this operation $j$
times will replace $F$ by a new $F$ that depends on a \emph{field at
scale $j$} characterised by the covariance $\sum_{k=j+1}^N C_k$.
According to estimates in \cite{BGM04}, this sum is dominated by the
first term which satisfies
\begin{equation}
    \label{e:scaling-estimate}
    |\nabla_{x}^\alpha \nabla_{y}^\beta C_{j+1} (x,y)|
    \le
    {\rm const} \, L^{-2j[\phi]-|\alpha |_{1}j-|\beta |_{1}j}
    ,
\end{equation}
where the symbol $[\phi]$, which is called the \emph{dimension} of the
field, was defined in \eqref{e:dimphidef}.  The typical field at scale
$j$ behaves like ``half a covariance,'' and in particular the standard
deviation of $\varphi_x$ is $\approx L^{-j[\phi]}$. Furthermore, the
estimate on derivatives in \eqref{e:scaling-estimate} says that
typical fields at scale $j$ are roughly constant over distances of
order $L^{j}$.

We can now explain why the terms in $V$ as defined by \eqref{e:Vdef}
play a pre-eminent role.  For a cube $B$ of side $L^{j}$, which
contains $L^{dj}$ points,
\begin{equation}
    \label{e:relevant-monomials}
    \sum_{x \in B} \varphi_{j,x}^p \approx L^{(d-p[\phi])j}
    .
\end{equation}
In the case of $d=4$, for which $[\phi]=1$, this scales down when
$p>4$ and $\varphi^{p}$ is said to be \emph{irrelevant}.  The power
$p=4$ neither decays nor grows, and is called \emph{marginal}.  Powers
$p<4$ grow with the scale, and are called \emph{relevant}.  Since the
derivatives in \eqref{e:scaling-estimate} provide powers of $L$, the
monomial $\varphi (-\Delta)\varphib$ is marginal. Thus $\tau
,\tau_{\Delta}, \tau^{2}$ are the supersymmetric marginal and relevant
monomials.

\subsection{Progressive integration}\label{sec:rg1}

To prove Theorem~\ref{thm:wsaw4rep} using \eqref{e:generating-fn} we
have to calculate
\begin{equation}
    \int_{\Cbold^{\Lambda}}
    e^{-S (\Lambda) - V_{0} (\Lambda)}
    =
    \Ebold_{C} e^{-V_{0} (\Lambda)}
    ,
\end{equation}
where $V_{0}$ is given by \eqref{e:V0def}. This $V_{0}$ equals $V$ as
defined in \eqref{e:Vdef}, with $(g,\nu ,z,\lambda,q)$ replaced by
$(g_{0},\nu_{0} ,z_{0},\lambda_{0},q_{0})$ with
\begin{equation}
    \label{e:initial-lq}
    q_{0}=0, \quad \quad \lambda_{0}=1
    .
\end{equation}

Sections~\ref{sec:gaussian-integration} and \ref{sec:decomposition}
have taught us that we can evaluate $\Ebold_{C} e^{-V_{0} (\Lambda)}$
by the following iteration: let
\begin{equation}
    \label{e:Z0def}
    Z_{0} = e^{-V_{0} (\Lambda)}.
\end{equation}
Inductively define $Z_{j},\, j=0,\dots ,N$, by
\begin{equation}
    Z_{j+1} = \Ebold_{C_{j+1}}\theta Z_{j}
    .
\end{equation}
Then
\begin{equation}
    \Ebold_{C}e^{-V_{0} (\Lambda)} =Z_{N}
    .
\end{equation}
Therefore the proof of Theorem~\ref{thm:wsaw4rep} now depends on the
analysis of the sequence $Z_{j}$.  Our proof will depend on showing
that the $Z_{j}$ simplify as $j$ increases. In fact, in the next
section we will see that they become more Gaussian, in the sense that
the $g\tau^{2}$ term becomes smaller. The index $j$ will be called a
\emph{scale}.

\section{Perturbation theory and flow equations}
\label{sec:perturb}

In this section we start to prove that $Z_{j}$ becomes more Gaussian
as $j$ increases. To do this we adapt to our particular setting a
perturbative calculation of the kind that appears in \cite{WiKo74}.

For $X \subset \Lambda$ and $V$ as defined in \eqref{e:Vdef},
define
\begin{equation}
    \label{e:Idef}
    I_{j,X}(V) = e^{-V(X)} \big( 1 + \half W_j(V,X) \big),
\end{equation}
where
\begin{equation}
    W_j(V,X) = (1-\LT_{X}) F_{w_j}\big(V(X),V(\Lambda)\big)
\end{equation}
with
\begin{align}
    w_j &= \sum_{i=1}^j C_i,
    \\
\label{e:Fexpand1bis}
    F_{w_j} \big(V(X),V(\Lambda)\big)
    &  =
    \sum_{n\ge 1} \frac{1}{n!}
    \big(D_{R}^n  V(X)\big) w_j^n \big( D_{L}^n   V(\Lambda)\big) ;
\end{align}
the latter sum truncates at $n=4$ due to our quartic interaction.  The
symbols $D_{R}$ and $D_{L}$ denotes right and left differentiation
with respect to fields.  The ``left/right'' is to specify signs, but
this and the precise definition are not of immediate importance,
so we just give an example.

\begin{example}
\label{ex:W}
For $V = \psi \psib$ and $X=\{x \}$, $\big(D_R^n V(X)\big) w_j^n \big(
D_L^n V(\Lambda)\big)$ equals
\begin{equation}
    \begin{cases}
    \sum_{y \in \Lambda}
    \Big(
    \psi_{x}
    w_{j} (x,y)
    \psib_{y}
    +
    \psib_{x}
    w_{j} (x,y)
    \psi_{y}
    \Big)
    &n=1
    \\
    -\sum_{y \in \Lambda} w_{j}^{2} (x,y)
    &
    n=2.
    \end{cases}
\end{equation}
\end{example}
When $j=0$, $I_{j,X}(V) = e^{-V (\Lambda)}$ because
$w_{0}=0$. Therefore we can choose the coupling constants to make it
equal to $Z_{0}$. Furthermore, $I_{j,X}(V)$ has the martingale-like
property exhibited in Proposition~\ref{prop:I-action}, which says that
integrating out the fluctuation field $\xi_{j+1}$ is approximately the
same as changing the coupling constants in $V$ to new coupling
constants called $(\gpt,\nupt,\zpt,\lambdapt,\qpt)$. The formulas for
the new coupling constants are called \emph{perturbative flow
equations}.

\begin{prop}\label{prop:I-action}
As a formal power series in $(g,\nu,z,\lambda,q)$,
\begin{equation}
    \label{e:I-invariance}
    \Ebold_{C_{j+1}} I_{j,\Lambda} (V)
    =
    I_{j+1,\Lambda} (\Vpt)
    \mod (g,\nu,z,\lambda,q)^{3},
\end{equation}
where
\begin{equation}
    \label{e:Vptdef}
    \Vpt = \Vpt (V)
\end{equation}
has the same form \eqref{e:Vdef} as $V$, with
$(g,\nu,z,\lambda,q)$ replaced by
\begin{align}
    &
    \label{e:recursion-g-pt}
    \gpt = g - c_{g} g^{2} +
    \rpt_{g,j},
    \\
    &
    \label{e:recursion-mu-pt}
    \nupt = \nu + 2gC_{j+1} (0,0) +
    \rpt_{\nu,j },\\
    &
    \label{e:recursion-z-pt}
    \zpt = z +
    \rpt_{z,j},\\
    &
    \label{e:obs-flow-lambda-pt}
    \lambdapt
    =
    \left(
    1 +
    g \delta_{j}
    +
    \frac{1}{2}
    \mu L^{-2j}\sum_{y\in\Lambda}C_{j+1} (0,y)
    \right)
    \lambda
    ,
    \\
    &
    \label{e:obs-flow-q-pt}
    \qpt
    =
    q
    +
    \lambda^{2} \, C_{j+1} (\pp,\qq)
    ,
\end{align}
where $c_{g}>0$ and
\begin{gather}
    \label{e:mu-def}
    \mu = L^{2j}\Big(\nu  + 2g \sum_{k=j+1}^{N}C_{k}(0,0)\Big)
    ,
\end{gather}
and where $\rpt_{g,j}$, $\rpt_{\mu,j}$ and $\rpt_{z,j}$ are
computable uniformly bounded homogeneous polynomials of degree $2$ in
$(g,\mu,z)$.  There are $g^{2}$ terms in $\rpt_{g,j}$, but they are
summable in $j$ and therefore do not overpower $c_{g}g^{2}$;
$\delta_{j}$ is a summable sequence of positive numbers determined by
$C_{j}$.
\end{prop}

\paragraph{The $\beta$ function.} The right hand side of
\eqref{e:recursion-g-pt} is known as the \emph{$\beta$ function}. The
simpler recursion obtained by setting $\rpt_{\nu,j }=0$, namely
\begin{equation}
\label{e:recursion-gbar}
    \gbar_{j+1} = \gbar_j - c_{g}\gbar_j^2,
    \quad\quad
    \gbar_0=g_{0},
\end{equation}
creates a sequence $\gbar_{j}$ that tends to zero like $j^{-1}$ as
$j\rightarrow \infty$.  The sequence $Z_{j}$
becomes more Gaussian due to the famous observation, known as
\emph{infra-red asymptotic freedom}, that \eqref{e:recursion-gbar}
controls the behaviour of the more complex recursion of
Proposition~\ref{prop:I-action} and drives the $\tau^2$ term to zero.

\section{The renormalisation group map}

The problem with the second order perturbative calculation in
Section~\ref{sec:perturb} is that the error is not only of order $3$
in the coupling constants. It is also at least of order $3$ in
$L^{-dj}|\Lambda|$.  In fact there will also be an $\exp (O
(L^{-dj}|\Lambda |))$ in this error!  We will call this the
\emph{$\exp (O (|\Lambda |))$ problem}.  The remedy is not to work
with $I_{j} (\Lambda)$, but with $\prod_{B\subset \Lambda}I_{j} (B)$
where $B$ is a cube and the allowed cubes pave $\Lambda$.  The idea is
that by choosing the side of $B$ to be bigger than the range of
$C_{j+1}$, we can take advantage of independence of cubes that do not
touch to more or less use our perturbation theory with $\Lambda$
replaced by individual cubes. This idea requires a systematic
organisation which we describe in this section.

\subsection{Scales and the circle product}

Let $L\ge 3$ be an integer.  Let $R=L^N$, and let $\Lambda = \Zd /
(R\Zd)$.

\begin{definition}
\label{def:blocks}
(a) \emph{Blocks.}
For each $j=0,1,\ldots,N$, the torus $\Lambda$ is paved in a natural
way by $L^{N-j}$ disjoint $d$-dimensional cubes of side $L^j$.  The cube
that contains the origin has the form (for $L$ odd)
\eq
    \{x\in \Lambda:  |x| \le \frac{1}{2} (L^{j}-1)\}
    ,
\en
and all the other cubes are translates of this one by vectors in
$L^{j} \Zd$.  We call these cubes $j$-{\em blocks}, or {\em blocks}
for short, and denote the set of $j$-blocks by ${\cal B}_j= {\cal
B}_j(\Lambda)$.
\\
(b) \emph{Polymers.}
A union of $j$-blocks is called a {\em polymer} or $j$-{\em
polymer}, and the set of $j$-polymers is denoted ${\cal P}_j={\cal
P}_j(\Lambda)$.  The size $|X|_j$ of $X\in {\cal P}_j$ is the number
of $j$-blocks in $X$.
\\
(c) \emph{Connectivity.}
A subset $X\subset \Lambda
$ is said to be
\emph{connected} if for any two points $x_{a}, x_{b}\in X$
there exists a path $( x_i,i=0,1,\dotsc n) \in X$
with $\|x_{{i+1}}-x_{i}\|_\infty =1$,
$x_{0} = x_{a}$ and $x_{n}=x_{b}$.  According to this definition, a
polymer can be decomposed into connected components; we write
$\Ccal(X)$ for the set of connected components of $X$.
We say that two polymers $X,Y$ \emph{do not touch} if
$\min\{\|x-y\|_\infty : x \in X, y \in Y\} >1$.
\\
(d)  \emph{Small sets.}
A polymer
$X\in \Pcal_{j}$ is
said to be a
\emph{small set} if $|X|_j \le 2^{d}$ and $X$ is
connected.
Let $\Scal_{j}$ be the set of all small sets in $\Pcal_{j}$. 
\\
(e) \emph{Small set neighbourhood.} 
For $X \subset \Lambda $ let
\begin{equation}
\label{e:9ssn}
    X^{*}
=
    \bigcup_{Y\in \Scal_{j}:X\cap Y \not =\varnothing } Y.
\end{equation}
\end{definition}

The \emph{polymers} of Definition~\ref{def:blocks} have nothing to do
with long chain molecules.  This concept has a long history in
statistical mechanics going back to the important paper
\cite{GrKu71}.

\begin{prop}\label{prop:factorisationE}
Suppose that $X_1,\ldots, X_n \in \Pcal_{j+1}$ do not touch each other
and let $F_i(X_i) \in \Ncal(X_i)$.  The expectation $\Ebold_{C_{j+1}}$
has the \emph{factorisation property}:
\begin{equation}
\label{e:Efaczz}
    \Ebold_{C_{j+1}} \prod_{i=1}^n F_i(X_i)
    =
    \prod_{m=1}^n \Ebold_{C_{j+1}} F_i(X_i).
\end{equation}
\end{prop}

\begin{proof}
Gaussian random variables are independent if and only if the
off-diagonal part of their covariance matrix vanishes.  This
generalises to our forms setting, and so the proposition follows from
the finite range property of $C_{j+1}$.
\end{proof}

Given forms $F,G$ defined on ${\cal P}_j$,
let
\begin{equation}
    (F \circ G)(\Lambda) = \sum_{X\in {\cal P}_j} F(X) G(\Lambda \setminus X).
\end{equation}
This defines an associative product, which is also commutative
provided $F$ and $G$ both have even degree.

\subsection{The renormalisation group map}
\label{sec:rg-map}

Recall that we have defined $I_{j,X} (V)$ in \eqref{e:Idef}. Given a
yet-to-be-constructed sequence $V_{j}$, for $X \in {\cal P}_j$, let
\begin{equation}
    I_{j} (X)
    =
    \prod_{B\in \Bcal_{j}}I_{j,B} (V_{j})
    .
\end{equation}
We have defined $V_{0}$ in \eqref{e:V0def}.  Let $K_0(X) =
\1_{X=\varnothing}$.  Then the $Z_{0}$ defined in \eqref{e:Z0def} is
also given by
\begin{equation}
    Z_0 = I_0(\Lambda) = (I_0 \circ K_0)(\Lambda)
    ,
\end{equation}
because $I_{0,\Lambda} (\Lambda)=e^{-V_{0} (\Lambda)}$ since
$w_0=0$.

\begin{definition}\label{def:factorisationK}
We say that $K : \Pcal_j \to \Ncal$ has the
\emph{component factorisation property}
if
\begin{equation}
\label{e:Kfac}
    K (X)
=
    \prod_{Y \in \Ccal( X)}K (Y).
\end{equation}
\end{definition}

Suppose, inductively, that we have constructed $(V_{j},K_{j})$ where
$K_{j} : \Pcal_j \to \Ncal$ is such that
\begin{align}
    \label{e:Kaxioms}
    \begin{split}
    (i)\quad&
    Z_{j} = (I_{j} \circ K_{j})(\Lambda),
    \\
    (ii)\quad&
    \text{$K_{j}$ has the component factorisation property},
    \\
    (iii)\quad&
    \text{For $X \in \Pcal_{j}$, $K_{j}(X) \in \Ncal (X^{*})$}.
    \end{split}
\end{align}
Our objective is to define $(V_{j+1},K_{j+1})$, where $K_{j+1} :
\Pcal_{j+1} \to \Ncal$ has the same properties at scale $j+1$.  Then
the action of $\Ebold_{C_{j+1}}\theta$ on $Z_{j}$ has been expressed
as the map:
\begin{equation}
    (V_j, K_j) \mapsto (V_{j+1},K_{j+1})
    .
\end{equation}
This map will be constructed next.  We call it the
\emph{renormalisation group map}. Unlike $Z_{j} \mapsto \Ebold \theta
Z_{j}$ it is not linear, so this looks like a poor trade, but in fact
it is a good trade because the data $(V_{j},K_{j})$ is local, unlike
creatures such as $\exp (-V_{j} (\Lambda))$ in $Z_{j}$.  The component
factorisation property and Proposition~\ref{prop:factorisationE}
allows us to work with $K_{j}$ on the domain of all connected sets in
$\Pcal_{j}$.  We can prove that $K_{j} (X)$ is very small when the
number of blocks in $X$ is large; in fact, only the restriction of
$K_{j}$ to the small sets $\Scal_{j}$ plays an important role.

\section{The inductive step: construction of $V_{j+1}$}
\label{sec:Vjplus1}

In accordance with the program set out in Section~\ref{sec:rg-map} we
describe how $V_{j+1}$ is constructed, given $(V_{j},K_{j})$. Our
definition of $V_{j+1}$ will be shown to have an additional property
that there is an associated $K_{j+1}$, which, as a function of
$K_{j}$, is contractive in norms described in Section~\ref{sec:norms}.

Recall that the set $\Scal$ of small sets was given in
Definition~\ref{def:blocks}.  For $B\in \Bcal_{j}$ define $V_{j+1}$ to
be the local interaction determined by:
\begin{align}
    \label{e:Vjplus1}
    \begin{split}
    &
    \Vhat_{j} (B)
    =
    V_{j} (B)
    +
    \LT_{B}
    \sum_{Y \in \Scal, Y \supset B} \frac{1}{|Y|}  I_{j} (Y)^{-1} K_{j} (Y)
    ,
    \\
    &
    V_{j+1}
    = \Vpt (\Vhat_{j})
    ,
    \end{split}
\end{align}
where $\Vpt = \Vpt (V)$ with generic argument $V$ is defined in
\eqref{e:Vptdef}.  Recalling the discussion of ``relevant terms'' just
after \eqref{e:relevant-monomials}, one sees in \eqref{e:Vjplus1} that
$V_{j+1}$ has been defined so that relevant terms that would expand if
they remained inside $K_{j}$ are being absorbed into $V_{j+1}$.

We have completed the $V$ part of the inductive construction of the
sequence $(V_{j},K_{j})$.  Before discussing the $K$ induction we have
to define some norms so that we can state the contractive property.

\section{Norms for $K$}
\label{sec:norms}

Let $\h_j>0$ and $\s_j>0$.  For a test function $f$ as defined in
Section~\ref{sec:lml} we introduce a norm
\begin{equation}
\label{e:Phignorm}
    \|f\|_{\Phi_j}
    = \sup_{x,y \in \Lambda^*, z \in \{1,2\}^*}
    \sup_{|\alpha|_{\infty} \leq 3}
    \h_j^{-p-q} \s_j^{-r} L^{j|\alpha|_{1}} |\nabla^\alpha f_{x,y,z}|.
\end{equation}
Multiple derivatives up to order $3$ on each argument are specified by
the multi-index $\alpha$.  The gradient $\nabla$ represents the
finite-difference gradient, and the supremum is taken componentwise
over both the forward and backward gradients.  A test function $f$ is
required to have the property that $f_{x,y,z}=0$ whenever the sequence
$x$ has length $p>9$ or the sequence $z$ has length $r>2$; there is no
restriction on the length of $y$. By the definition of the norm, test
functions satisfy
\begin{equation}
\label{e:phisizebd}
    |\nabla^\alpha f_{x,y,z}|
    \leq
    \h_j^{p+q} \s_j^r L^{-j|\alpha|_{1}}\|f\|_{\Phi_j}.
\end{equation}
We discuss the choice of $\s_j$ in Section~\ref{sec:focc} when it
first plays a role, and here we focus on $\h_j$.  An important choice
is
\begin{equation}\label{e:ljdef}
    \h_j=\ell_{j}=\ell_{0}L^{-j[\phi]}
    ,
\end{equation}
for a given $\ell_0$.  The $L^{-j[\phi]}$ is there because unit norm
test functions of one variable should obey the same estimates as a
typical field, and test functions of more than one variable should
obey the estimates that a product of typical fields obeys.

Recall the pairing defined in \eqref{e:pairing} and, for $F \in
\Ncal$ and $\phi \in \C^\Lambda$, let
\begin{equation}
\label{e:Tdefnew}
    \|F\|_{T_{\phi,j}}
    =
    \sup_{g : \|g\|_{\Phi_j} \le 1}
    \left|\pair{F,g}_{\phi}\right|
    .
\end{equation}
The following proposition provides properties of this seminorm that
are well adapted to the control of $K$.

\begin{prop}
Let $F,F_1,F_2 \in \Ncal$.
The $T_\phi$ norm obeys the product property
\eq
    \|F_{1} F_{2} \|_{T_{\phi,j}}
    \le
    \|F_{1}\|_{T_{\phi,j}} \| F_{2} \|_{T_{\phi,j}},
\en
and, if $\ell_0$ is chosen large enough, the integration
property
\eq
    \|\Ebold_{C_{j+1}} F\|_{T_{\phi,j}(\h_j)}
    \le \Ebold_{C_{j+1}} \|F\|_{T_{\phi+\xi,j}(2\h_j)}.
\en
\end{prop}

For further details, see \cite{BS11}.  The second conclusion shows
that the norm controls the forms when a fluctuation field is
integrated out: on the right hand side the norm is a zero degree form,
and hence the expectation is a standard Gaussian expectation.

The most important case of the $T_\phi$ seminorm is the case $\phi
=0$, but knowing that $\|K (X)\|_{T_{0}}<\infty$ cannot tell us
whether $K (X)$ is integrable. For this we must limit the growth of $K
(X)$ as $\phi \rightarrow \infty$, and the resolution of this issue
will be obtained using Definition~\ref{def:regulator} below.

Our intuitive picture of $K_{j} (X)$, where $X \in \Pcal_{j}$, is that
it is dominated by a local version of the remainder
$(g,\nu,z,\lambda,q)^{3}$ in \eqref{e:I-invariance}.  To estimate such
remainders we must, in particular, estimate $I_{j,X}$ which contains
$\exp (-g_{j}\sum_{x\in X}|\varphi_x|^4)$.  By
\eqref{e:scaling-estimate} the typical field $\varphi$ at scale $j$ is
roughly constant on scale $L^{j}$, and $X$ contains $O (L^{jd})$
points. Therefore this factor looks like
$\exp(-g_{j}L^{dj}|\varphi|^4)$.  This is a function of
$\varphi/\h_{j}$ with $\h_{j} \approx g_{j}^{-1/4}L^{-jd/4}$, which in
four dimensions can be rewritten as $g_{j}^{-1/4}L^{-j[\phi]}$ because
$[\phi]=1$.  We want to prove that $g_{j}$ decays in the same way as
does $\gbar_{j}$ in \eqref{e:recursion-gbar}, and with this in mind we
replace $g_j$ by the known sequence $\gbar_j$.  This leads us to our
second choice
\begin{equation*}
    \h_j = h_j = k_0\gbar_j^{-1/4}L^{-j[\phi]}
    ,
\end{equation*}
where the constant $k_{0}$ is determined so that $\exp (-V_{j} (B))$ will,
uniformly in $j$, have a $T_{\phi} (h_{j})$ norm close to one.

In the previous discussion we made the assumption that the typical
$\varphi$ at scale $j$ is roughly constant on scales $L^{j}$.  Our
norm recognises this; it is a weighted $L_{\infty}$ norm, where the
weight permits growth as fields become atypical.  The weight is
called a large field regulator and is defined next.

Consider a test function $f$ that is an ersatz field $\varphi$, namely
a complex-valued function
$f=f_x$ for $x \in \Lambda$.  For $X \subset \Lambda$, we write $f \in
\Phipol(X)$ if $f$ restricted to $X$ is a polynomial of degree three
or less.  We define a seminorm on $\phi = (\varphi,\varphib)$ by
\begin{equation}
\label{e:9Phitilnorm}
    \|\phi \|_{\tilde\Phi_j(X)}
    =
    \inf
    \{
    \|\varphi - f\|_{\Phi_j(\ell_j)} : f \in \Phipol(X)
    \};
\end{equation}
note that we are setting $\h_j=\ell_j$ in the above equation.

\begin{definition}
\label{def:regulator}
Let $j \in \N_{0}$, $X \in \Pcal_j$, and $\phi \in \C^\Lambda$.
The \emph{large-field regulator} is given  by
\begin{align}
\label{e:9Gdef}
    \tilde G_{j} (X,\phi)
    =
    \prod_{B \in \Bcal_j(X)} \exp  \|\phi\|_{\tilde\Phi_{j}(B^*)}^2
    ,
\end{align}
where $B^{*}$ is the small set neighbourhood of $B$ defined in
\eqref{e:9ssn}.  For each $X \in \Pcal_j$, we define a seminorm on
$\Ncal(X^{*})$ as follows.  For $K(X) \in \Ncal(X^{*})$, we define
$\|K(X)\|_{\tilde{G}_{j},h_{j}}$ to be the best constant $C$ in
\begin{equation}
\label{e:9Nnormdef}
    \|K(X)\|_{T_{\phi,j} (h_{j})}
    \le
    C
    \tilde{G}_{j}(X,\phi)
    ,
\end{equation}
where we have made explicit in the notation the fact that the norm on
the left hand side is based on the choice $\h_j=h_j$.
\end{definition}

\section{The inductive step completed: existence of $K_{j+1}$}
\label{sec:rgstep}

We have already specified $V_{j+1}$ in \eqref{e:Vjplus1}.
Now we complete the inductive step by constructing $K_{j+1}$ such
that \eqref{e:Kaxioms} holds. The following
theorem is at the heart of our method \cite{BS11}.
It provides $K_{j+1}$ and says that we can continue to
prolong the sequence $(V_{j},K_{j})$ for as long as the coupling
constants $(g_{j},\nu_{j},z_{j})$ remain small.  Moreover, in this
prolongation, the $T_{0}$ norm of $K_{j+1}$ remains third order in the
coupling constants and is therefore much smaller than the perturbative
($K$-independent) part of $V_{j+1}$.

For $a\ge 0$, set $f_{j} (a ,\varnothing)=0$, and define
\begin{equation}
    \label{e:fdef}
    f_{j} (a ,X)
    =
    3 + a (|X|_j-2^d)_+,
    \quad\quad
    \text{$X \in \Pcal_{j}$ with $X \neq \varnothing$}.
\end{equation}
Note that $f_{j} (a ,X)=3$ when $X \in \Scal_j$, but that $f_{j} (a
,X)$ is larger than $3$ and increases with the size of $|X|_j$ if $X
\nin \Scal_j$.  We fix $a$ to have a sufficiently small positive
value.

The following theorem is proved for two different choices of the norm
pairs $\|\cdot\|_j$ and $\|\cdot\|_{j+1}$, in \refeq{KIH} and
\refeq{nextKbound}, and for two corresponding choices of the small
parameter $\epsilon_{\delta I}$, as follows:
\begin{itemize}
\item
$\|\cdot\|_j =\|\cdot\|_{\tilde{G}_{j},h_{j}}$ with $h_{j}=
k_{0}\bar g^{-1/4}_{j}L^{-j[\phi]}$,
and
$\|\cdot\|_{j+1} =\|\cdot\|_{\tilde{G}_{j+1},h_{j+1}}$ with $h_{j+1}=
k_{0}\bar g^{-1/4}_{j+1}L^{-(j+1)[\phi]}$.
The small parameter
$\epsilon_{\delta I}$ is proportional to $g_j^{1/4}$.
\item
$\|\cdot\|_j =\|\cdot\|_{T_{0}, \ell_j}$ with $\ell_j = \ell_0L^{-j[\phi]}$,
and $\|\cdot\|_{j+1} =\| \cdot \|_{T_0,\ell_{j+1}}$.
The small parameter
$\epsilon_{\delta I}$ is proportional to $g_j$.
\end{itemize}
Define a cone $C = \{(g_{j},\nu_{j},z_{j}) |g>0, \,|\nu | \vee |z| \le
b g, \,g_{j} \le c (b,L)\}$.  The constant $b$ is arbitrary, and $c
(b,L)$ is a function of $b,L$ constructed in the proof of the next
theorem.

\begin{theorem}
\label{thm:1} Let $(g_{j},\nu_{j},z_{j}) \in C$.  Let $a$ be
sufficiently small, and let $M$ be any (large) positive constant that
is independent of $d,L$.  There is a constant $c_{\rm pt}$ (depending
on $d,L$) such that the following holds.  Suppose that $K_{j}: \Pcal_j
\to \Ncal_j$ has properties \eqref{e:Kaxioms} and satisfies
\begin{equation}
\label{e:KIH}
    \|K_{j} (X)\|_{j}
    \le
    Mc_{\rm pt}
    \epsilon_{\delta I}^{f_{j} (a ,X)},
    \quad \quad
    X \in \Pcal_{j} \ \text{connected}
    ,
\end{equation}
Then, if $L$ is sufficiently large (depending on $M$), there exists
$K_{j+1}: \Pcal_{j+1} \to \Ncal_{j+1}$ with properties
\eqref{e:Kaxioms} at scale $j+1$
and
\begin{align}
    \label{e:nextKbound}
    \|K_{j+1} (U)\|_{j+1}
    &\le
    2c_{\rm pt}
    \epsilon_{\delta I}^{f_{j+1} (a ,U)},
    \quad\quad
    U \in \Pcal_{j+1} \ \mathrm{connected}
    .
\end{align}
\end{theorem}

\section{Decay of the two-point function}
\label{sec:focc}

Finally, we combine the machinery we have developed, to outline the
proof of Theorem~\ref{thm:wsaw4rep}.  As we have already noted,
Theorem~\ref{thm:wsaw4} is a consequence of
Theorem~\ref{thm:wsaw4rep}.

We must study the coupling constant flow.  The linear map
$\LT_{B}:\Ncal \rightarrow \Vcal$ is bounded in $T_{0}$ norm
\cite{BS11}, so according to the inductive assumption \refeq{KIH} on
the $T_0$ norm of $K_j$, the coupling constants in $\Vhat_j$ of
\eqref{e:Vjplus1} are small (third order) adjustments to the coupling
constants in $V_j$.  Theorem~\ref{thm:1} ensures that this smallness
is preserved as the scale advances.

We first consider the case $(\lambda_{0},q_{0})= (0,0)$.  In this
case, $(\lambda_{j} ,q_{j})= (0,0)$ for all $j$.  The definition of
$V_{j+1}$ in \eqref{e:Vjplus1} then gives rise to a non-perturbative
version of the flow equations of Proposition~\ref{prop:I-action}, in
which the effect of $K$ is now taken into account.  When $V_{j}\mapsto
V_{j+1}$ is expressed as
\begin{equation}
    (g_{j},\nu_{j},z_{j}) \mapsto (g_{j+1},\nu_{j+1},z_{j+1})
\end{equation}
we find that
\begin{align}
    &
    \label{e:recursion-g}
    g_{j+1} = g_{j} - c_{g} g_{j}^{2} +
    r_{g,j},
    \\
    &
    \label{e:recursion-mu}
    \nu_{j+1} = \nu_{j} + 2gC_{j+1} (0,0) +
    r_{\nu,j },\\
    &
    \label{e:recursion-z}
    z_{j+1} = z_{j} +
    r_{z,j},\\
    &
    \label{e:recursion-K}
    K_{j+1} = r_{K,j} (g_{j},\nu_{j},z_{j},K_{j})
    ,
\end{align}
where the $r$'s now depend also on $K_{j}$, and where we have added
the map $r_{K,j}: (g_{j},\nu_{j},z_{j}, K_{j}) \mapsto K_{j+1}$
defined by Theorem~\ref{thm:1}.  Furthermore, we prove that the $r$'s
are Lipschitz functions of $(g_{j},\nu_{j},z_{j},K_{j})$, where $K$
belongs to a Banach space normed by a combination of the norms in
Section~\ref{sec:rgstep}.  These are the properties needed to prove
that $K$ only causes a small deformation of the perturbative flow $V
\mapsto \Vpt$.

The main theorem now reduces to an exercise in dynamical systems. We
prove that there is a \emph{Lipschitz stable manifold} of initial
conditions $(z_{0},\nu_{0}) = h (m^{2},g_{0})$ for which the sequence
$(V_{j},K_{j}),\, j=0,\dots ,N$, has a limit as $N\rightarrow \infty$
and $m^{2}\downarrow 0$.  We call this the \emph{global
trajectory}. For $m^{2}=0$, the global trajectory tends to the fixed
point $(V,K)=(0,0)$.  In particular, $g_{j} \rightarrow 0$, which is
infra-red asymptotic freedom.  Referring to \eqref{e:zmdef}, we have
four unknown parameters $g_{0},\nu_{0},z_{0},m^{2}$ related by three
equations, and now there is a fourth equation $(z_{0},\nu_{0}) = h
(m^{2},g_{0})$. By the implicit function theorem we solve for the
unknowns as functions of $(g,\nu)$. As $\nu \downarrow \nu_{c} (g)$,
$m^{2} \downarrow 0$ and vice-versa.

Now we consider the flow for $(\lambda_{j},q_{j})$. According to
\eqref{e:initial-lq}, $\lambda_{0}=1$ and $q_{0}=0$.  Using
\eqref{e:LTdef}, we prove that the terms $r_{g,j},r_{\nu ,j},r_{z,j}$
do not depend on $\lambda_{j},q_{j}$ and thus the coupling constants
$g,\nu,z$ have no dependence on $\lambda,q$. From \eqref{e:Vjplus1} we
find
\begin{align}
    &
    \label{e:obs-flow-lambda}
    \lambda_{j+1}
    =
    \left(
    1 +
    g_{j} \delta_{j}
    +
    \frac{1}{2}
    \mu_{j} L^{-2j}C_{j+1}^{(1)}
    \right)
    \lambda_{j} + r_{\lambda ,j}
    ,
    \\
    &
    \label{e:obs-flow-q}
    q_{j+1}
    =
    q_{j}
    +
    \lambda^{2} \, C_{j+1} (\pp,\qq) + r_{q,j}
    ,
\end{align}
where $r_{\lambda,j},r_{q,j}$ are corrections arising from $K_{j}$.

Recall that $\Scal_{j}$ was defined in Definition~\ref{def:blocks}.
Let $s_{a,b}$ be the first scale $j$ such that there exists a polymer
in $\Scal_{j}$ that contains $\{\pp ,\qq \}$.  The correction
$r_{q,j}$ is zero for all scales $j<s_{a,b}$: according to
\eqref{e:LTdef} and the definition of $\Vhat $ in \eqref{e:Vjplus1}
there can be no $\sigma \sigmab$ contribution from $K_{j}$ until the
first scale where there is a set $X \in \Scal_{j}$ that covers $\{\pp
,\qq \}$. Also, by the finite range property, $C_{j+1} (\pp ,\qq)=0$
for $j<s_{a,b}$. Thus \eqref{e:obs-flow-q} gives
\begin{equation}
    \label{e:q-infty}
    q_{N}
    =
    \sum_{j=s_{\pp ,\qq}}^{N}
    \left(
    \lambda_{j}^{2} \,C_{j+1}(\pp,\qq)
    +
    r_{q,j}
    \right)
    .
\end{equation}
At scale $N$, $\Lambda$ is a single block in $\Bcal_{N}$, so by the
definition of the circle product, $Z_N$ is simply given by
\begin{equation}
    Z_{N}
    =
    (I_{N}\circ K_{N}) (\Lambda)
    =
    I_{N} (\Lambda) + K_{N}(\Lambda)
    .
\end{equation}
The final renormalisation group map is the action of $\Ebold_{C_{N}}$,
not $\Ebold_{C_{N}}\theta $.  This means that the fields $\phi ,\psi$
are to be set to zero in $I_{N},K_{N}$, and only dependence on
$\sigma$ remains. By \eqref{e:Idef} we compute two $\sigma$
derivatives of $I_{N}$ and find
\begin{align}
    -
    \left.
    \frac{\partial^{2}}{\partial \sigmaa\partial \sigmab}
    \right|_0
    Z_{N}
    =
    q_{N}
    -
    K_{\sigmab \sigma },
    \quad \quad \text{where} \quad
    K_{\sigmab \sigma }
    =
    \left.
    \frac{\partial^{2}K_{N}(\Lambda)}{\partial \sigmaa\partial \sigmab}
    \right|_0
    .
\end{align}
The $\sigmab \sigma $ derivative is a coefficient in the pairing
\eqref{e:pairing}, and the $T_{0}$ norm bounds this pairing, so
Theorem~\ref{thm:1} gives
\begin{equation}
    |K_{\sigmab \sigma }|
    \le
    \|K\|_{T_{0}} \s_N^{-2}
    \le
    O (g_{j}^{3})\s_N^{-2}
    .
\end{equation}
We are able to prove Theorem~\ref{thm:1} with
\begin{equation}
\label{e:sfrak}
    \s_j =
    \s_0 \ell_{j\wedge s_{\pp,\qq}}^{-1}
    \approx
    O (L^{j\wedge s_{a,b}})
    ,
\end{equation}
where $\s_0$ is a constant, so that, when $N > s_{a,b}$,
\begin{equation}
    \label{e:Ksigma-bound}
    |K_{\sigmab \sigma }|
    \le
    O (g_{N}^{3}) L^{-2N\wedge s_{a,b}} = O(g_N^3|a-b|^{-2})
    .
\end{equation}
This tends to zero as $N \rightarrow \infty$. 

By a similar estimate we can control the $r_{\lambda ,j},\,r_{q,j}$
terms in \eqref{e:obs-flow-lambda}, \eqref{e:q-infty}.  These contain
$\sigma$ derivatives of the $K_{j}$ terms in \eqref{e:Vjplus1}.  The
conclusion is that $\lambda_{\infty} = \lim_{N\rightarrow \infty
}\lambda_{N}$ and $q_{\infty} = \lim_{N\rightarrow \infty }q_{N}$
exist and are bounded away from zero.

By \eqref{e:generating-fn}, the left hand side of \refeq{eta0} is
given by
\begin{align}
\label{e:final}
    \lim_{\nu \downarrow \nu_{c}}
    (1+z_{0})
    \lim_{\Lambda \uparrow \Zd}
    \int_{\Cbold^{\Lambda}}
    e^{-S (\Lambda) - \tilde{V}_{0} (\Lambda)}
    \bar\varphi_{\pp} \varphi_{\qq}
    &=
    \big(\lim_{\nu \downarrow \nu_{c}} (1+z_{0})\big)
    \lim_{m^{2}\downarrow 0}
    q_{\infty }
    .
\end{align}
From \eqref{e:obs-flow-lambda} and \eqref{e:q-infty} we find that
\begin{align}
    \lim_{m^{2}\downarrow 0}q_{\infty }
    &\sim
    \lambda_{\infty}^{2}
    \sum_{j=s_{a,b}}^\infty C_{j+1} (\pp ,\qq)
    ,
\end{align}
where $m^{2}=0$ in $C_{j+1}$, and $\sim$ means that the ratio of the
left hand side and the right hand side tends to one as $\pp -\qq
\rightarrow \infty$. Next, we use the finite range property to restore
the scales $j< s_{a,b}$ to the sum, which then becomes the complete
finite range decomposition for the infinite volume simple random walk
two-point function $(-\Delta)^{-1}(\pp ,\qq)$,
\begin{align}
\label{e:qinftylim}
    \lim_{m^{2}\downarrow 0}q_{\infty }
    &\sim
    \lambda_{\infty}^{2} (-\Delta)^{-1} (\pp,\qq)
    .
\end{align}
The right hand side of \eqref{e:qinftylim}, and hence of
\refeq{final}, is thus asymptotic to a multiple of $|\pp -\qq|^{-2}$
as $|\pp - \qq| \rightarrow \infty$, as desired, since the inverse
Laplacian has this behaviour.

\section*{Acknowledgements}

We thank Roland Bauerschmidt for contributions to the proof of
Proposition~\ref{thm:Glims}.  The work of both authors was supported
in part by NSERC of Canada.  DB gratefully acknowledges the support
and hospitality of the Institute for Advanced Study and Eurandom,
where part of this work was done, and dedicates this work to his wife
Betty Lu.  GS gratefully acknowledges the support and hospitality of
the Institut Henri Poincar\'e, and of the Kyoto University Global COE
Program in Mathematics, during stays in Paris and Kyoto where part of
this work was done.


\end{document}